\documentclass{amsproc}

\usepackage{amsmath}
\usepackage{amssymb}
\usepackage[mathcal]{eucal}
\usepackage{rotating}
\usepackage{graphicx}
\usepackage[tableposition=below]{caption}
\usepackage{pigpen}
\usepackage{multirow}
\usepackage{hyperref}

\DeclareGraphicsExtensions{.png,.pdf,.eps}
\usepackage[all,2cell,cmtip]{xy}
\usepackage{tikz}
\usepackage{color}
\usepackage{bm}
\usepackage{longtable}
\usepackage{soul}
\newtheorem{theorem}{Theorem}[section]
\newtheorem{lemma}[theorem]{Lemma}
\newtheorem{corollary}[theorem]{Corollary}
\newtheorem{proposition}[theorem]{Proposition}

\newtheorem{notation}[theorem]{Notation}

\theoremstyle{definition}

\newtheorem{definition}[theorem]{Definition}

\newtheorem{problem}{Problem}
\newtheorem{remark}[theorem]{Remark}

\newtheorem{example}[theorem]{Example}

\def\C{{\mathbb C}}

\def\P{{\mathbb P}}
\def\Q{{\mathbb Q}}
\def\R{{\mathbb R}}
\def\Z{{\mathbb Z}}

\def\cD{{\mathcal D}}

\def\cO{{\mathcal{O}}}

\def\Q{{\mathbb{Q}}}

\def\operatorname#1{\mathop{\rm #1}\nolimits}

\def\Pic{\operatorname{Pic}}

\def\deg{\operatorname{deg}}

\def\Nef{{\operatorname{Nef}}}
\def\Amp{{\operatorname{Amp}}}

\def\NU{{\operatorname{N^1}}}

\newcommand{\pb}{\ar@{}[dr]|(.50){\text{\pigpenfont J}}}

\makeatletter
 
\newcommand*\wthelper[2]{%
        \hbox{\dimen@\accentfontxheight#1%
                \accentfontxheight#11.15\dimen@
                $\m@th#1\widetilde{#2}$%
                \accentfontxheight#1\dimen@
        }%
}

\newcommand*\accentfontxheight[1]{%
        \fontdimen5\ifx#1\displaystyle
                \textfont
        \else\ifx#1\textstyle
                \textfont
        \else\ifx#1\scriptstyle
                \scriptfont
        \else
                \scriptscriptfont
        \fi\fi\fi3
}
\makeatother

\makeindex

\begin{document}

\title[Discriminants of Toric Varieties]{Discriminants of Toric Varieties}

\author[R. Mu\~noz]{Roberto Mu\~noz}
\address{Departamento de Matem\'atica Aplicada, ESCET, Universidad
Rey Juan Carlos, 28933-M\'ostoles, Madrid, Spain}
\email{roberto.munoz@urjc.es}
\thanks{Partially supported by the spanish government project MTM2015-65968-P}
%
\author[A. Nolla]{\'Alvaro Nolla}
\address{Departamento de Did\'acticas Espec\'ificas, Facultad de Formaci\'on del Profesorado y Educaci\'on, Universidad Aut\'onoma de Madrid, 28049-Madrid, Spain} \thanks{}
\email{alvaro.nolla@uam.es}


\begin{abstract} We study the subvariety of singular sections, the discriminant, of a base point free linear system $|L|$ on a smooth toric variety $X$. On one hand we describe pairs $(X,L)$ for which the discriminant is of low dimension. Precisely, we collect some bounds on this dimension and classify those pairs whose dimension differs the bound less than or equal to two. On the other hand we study the degree of the discriminant for some relevant families on polarized toric varieties, describing their minimal values and the region of the ample cone where this minimal is attained. 
\end{abstract}

\maketitle

\section{Introduction}\label{sec:intro}
Let $X$ be a smooth irreducible complex projective variety of dimension $n$ and $L$ a line bundle on $X$. The discriminant $\mathcal{D}(X,L)$ (just $\mathcal{D}$ when $X$ and $L$ are clear from the context) of the corresponding linear system $|L|$ is defined as the subvariety of the projective space $|L|$ of its singular elements: 
$$
\cD(X,L)=\{H \in |L|:\; H \hbox{ is singular}\} \subseteq |L|
$$

When $L$ is very ample the discriminant is irreducible and is classically known as the dual variety of the linearly normal embedded variety $X \subseteq |L|^\vee$. When $|L|$ is just base point free, by Bertini Theorem, the discriminant $\cD \subset |L|$ is a subvariety of $|L|$ of positive codimension, not necessarily irreducible. Two main invariants of  $\cD \subseteq |L|$ have been classically studied: codimension and degree, especially in the case $L$ very ample, see \cite{Tev} and references therein, but also in more general contexts, like $L$ ample and base point free, \cite{LPS0} or just base point free \cite{LM1}.

The main ingredient in the study of degree and codimension of discriminant varieties is the set of Chern classes of the so called first jet bundle $J_1(L)$ of $L$, which encodes some first order infinitesimal information of the linear system. The vanishing of its Chern class of order $n$ is related to the fact that the codimension of the discriminant is bigger than one, and the degree with respect to $L$ of the maximal order non-zero Chern class is related to the degree of its discriminant (see Lemma \ref{lem:chern}).

When $X$ is a toric variety, that is, when it contains a Zariski open subset isomorphic to a torus $\C^{*n}$ such that the action of the torus extends to an action on $X$, the presence of the generalized Euler sequence (see (\ref{eq:euler}) in Section \ref{sec:toric}) leads to explicit formulae for the Chern classes of the first jet bundle of a line bundle $L$. In particular, one can use these formulae to study when the Chern classes vanish or to bound their degree in order to face questions of classification of pairs $(X,L)$, $X$ toric, $L$ globally generated line bundle, such that the dimension of the discriminant is small, or its degree is low. 

In relation to the first question, the classification of $(X,L)$ whose discriminant has small dimension, when $L$ is very ample  Di Rocco has shown in \cite{dirocco} that if the discriminant is not a hypersurface then $X$ is a decomposable (that is, the associated vector bundle is a sum of line bundles) projective bundle over a smooth toric variety. Our first goal in this paper is to study this question when $L$ is just globally generated, see Problem \ref{problem}. The main idea in Di Rocco's proof is to use the particular geometry of $X \subset |L|^\vee$, which is fibered in toric varieties of Picard number one, finally projective spaces. The case in which $L$ is just globally generated is different because the differential of the map $\phi_L$ defined by $L$ is not necessarily of maximal rank at any point, so that not all the first order infinitesimal information of the linear system $|L|$ can be studied in the image of $X$ by $\phi_L$.  In particular, the discriminant could be reducible. On one hand, the dual variety of $\phi_L(X) \subseteq |L|^\vee$ is contained in the discriminant of $(X,L)$ and, on the other hand, some extra singular sections in $|L|$ can appear due to droppings in the rank of the differential of $\phi_L$. 

Putting together this information, certain bounds on the dimension of the discriminant can be settled, see Lemma \ref{lemma:lm1summary}. In fact, the codimension of $\cD$ in $|L|$ is smaller than or equal to $\dim X+1$ and to $\dim |L|+1$, and we deal with the question of classifying pairs $(X,L)$ reaching the bound, and the two subsequent cases. The description of the geometry of $\phi_L(X)$, see \cite{FI}, and the particular behaviour of linear sytems on toric varieties allow us to precisely describe pairs $(X,L)$, $X$ smooth toric variety, $L$ globally generated such that $k:=\dim |L|-1-\dim \cD$ is equal to $n,n-1,n-2$ or $\dim |L|, \dim |L|-1, \dim |L|-2$ (see Lemma \ref{lema:bound}, Proposition \ref{prop:notbig},  Theorem \ref{thm:big} and Proposition \ref{prop:N-2}). If the dimension of $|L|$  is not too small, mainly toric fibrations or scrolls appear, resembling, in some sense, the behaviour of the very ample case in these extremal cases.

In relation to the second question, the classification of pairs $(X,L)$ whose discriminant has low degree, there exist classical results for $L$ very ample, see \cite[Thm.~5.2]{zak}, of classification up to degree three.  When toric, they are the projective space, some products of projective spaces, or some rank one projective bundles over projective spaces, with particular choices of $L$ (see Theorem \ref{thm:codegree}). It sounds natural to provide explicit formulae for the degree of the discriminant of this kind of toric varieties, looking for examples of degree four and for general bounds for the degree of their discriminants.  In Section \ref{sec:examples} we produce explicit formulae for some Chern classes of the first jet bundle of these, as well as for other relevant examples. It is of interest to remark that, when fixing $X$, we can consider the degrees of the Chern classes of order $n$ of the first jet bundle of the different line bundles over $X$ as an integer function whose domain is the set of integer points of the nef cone of $X$. We have shown, see Proposition \ref{prop:mindeg}, that when restricting to the ample cone, this function is lower bounded and we give the explicit region where to find the minimal value. This fact allows us to read the explicit formulae of the particular examples showing the cases in which the degree of the discriminant is $4$  (see Propositions \ref{prop:rankone} and \ref{prop:rank2}). As a further illustration, we used the software {\em Macaulay2} \cite{M2} to study in detail two families of examples: rank one projective bundles over Hirzebruch surfaces and smooth Fano threefolds, presenting the lowest values of the degrees of their dual varieties (see Example \ref{ex:lbhirz}, Propositions \ref{prop:hirz} and \ref{prop:fano} and Figure \ref{fig:fano3}).

\medskip
\noindent {\bf{Acknowledgments.}} We would like to thank Javier Pello who suggested us a proof of formula (\ref{eq:products}).

\section{Preliminaries}\label{sec:prel}

\subsection{Generalities}

Let $X$ be an irreducible  smooth  projective variety of dimension $n$ and $L$ a line bundle on $X$. Along the paper the base field will be $\C$. As said in the introduction, the discriminant $\mathcal{D}(X,L)$ of the corresponding linear system $|L|$ is defined as the subvariety of the projective space $|L|$ of its singular elements. For $x \in X$ we will denote by $|L-x|$ the sublinear system of elements in $|L|$ passing through $x$ and by $|L-2x|$ those which are singular at $x$. Let us define the {\em discriminant defect} $k(X,L)$ (just $k$ when $X$ and $L$ are clear from the context) as $$k(X,L)=\dim|L|-1-\dim\cD(X,L),$$ which, as commented before, is non-negative when $L$ is base point free, by Bertini Theorem.

The possible positivity $k>0$ is reflected in the vanishing of some Chern class of the so called first jet bundle $J_1(L)$ of $(X,L)$, see \cite[Lem.~03]{LM1}. In fact, when $k>0$, the codimension of the discriminant is bigger than one so that it is possible to choose a line $\ell$ in $|L|$ not meeting $\cD$. The first jets of the elements of this line provide an exact sequence of this type:
$$0 \to \mathcal{O}_X^{\oplus 2} \to J_1(L) \to Q \to 0,$$ where $Q$ is a vector bundle (here the smoothness of any element in $\ell$ is used) of rank equal to $n-1$. This proves that $c_{n}(J_1(L))=0$.

If on the contrary $\cD$ is a hypersurface, then any line $\ell \subset |L|$ is meeting $\cD$. Suppose that for any $H \in \ell \cap \cD$ its singular locus is zero dimensional, then the degeneration locus of the map $\mathcal{O}_X^{\oplus 2} \to J_1(L)$ (a zero dimensional scheme) defined as before is representing the Chern class $c_{n}(J_1(L))$, \cite[Cor.~2.6]{LPS0}. We will identify these zero cycles representing the $n$-th Chern classes with their degrees. This, see \cite{LMjpaa}, implies that in this case $c_n(J_1(L))$ is the sum of the Milnor numbers $\mu_H(x)$ of the different singular points $x$ of the divisors $H \in \ell \cap \cD$.

We can summarize these known results in the following:

\begin{lemma}\label{lem:chern} Let $X$ be an irreducible smooth projective variety of dimension $n$ and $L$ a base point free line bundle on $X$. Denote by $J_1(L)$ the first jet bundle of $(X,L)$.
\begin{itemize}
\item[(i)] If $k>0$ then $c_{n}(J_1(L))=0$.
\item[(ii)] If $k=0$ and for a general line $\ell \subset |L|$ the singular locus of any $H \in \ell \cap \cD$ is zero dimensional then $c_{n}(J_1(L))=\Sigma \mu_H (x) >0$, the sum running along the singular points $x$ of the sections $H \in \ell \cap \cD$  and $\mu_H(x)$ standing for the Milnor number of $x$.
\end{itemize}
\end{lemma}

In particular, when $|L|$ is very ample and the dual variety $X^\vee \subset |L|$ is a hypersurface, since the general singularity of an element of $\cD$ is shown to be ordinary quadratic, then $c_n(J_1(L))$ is just the degree of the dual variety $X^\vee \subset |L|$. When the dual variety is not a hypersurface, $c_n(J_1(L))=0$. Moreover, in this last case, the defect $k=\dim |L|-1-\dim(X^\vee)$ is positive and the singular locus of the hyperplane section corresponding to any smooth point $H \in X^\vee$  is shown to be (also when $X \subset \P^N$ is singular, where the dual is defined as the closure of hyperplanes tangent to $X$ at smooth points of $X$) a linear space of dimension $k$, see \cite[Thm.~1.18]{Tev}, called the {\it contact locus} of $X$ and $H$. Hence, through a general point of $X$, there exists a linear space of dimension $k$ contained in $X \subset |L|^\vee$. Such property of $X \subset |L|^\vee$ is usually said as being {\it ruled in linear spaces of dimension $k$}. Then, considering a general $\P^k$ in $|L|$, it is meeting the dual variety in a finite set of points, whose singular locus for any of them is a linear space. This allow to compute the degree of the dual variety in the following way:

\begin{remark}(\cite[Rmk.~1.6.11]{BSadjunction})\label{rem:degdual} For $X$ irreducible smooth projective variety, $L$ a very ample line bundle providing an embedding $X \subset|L|^\vee$, the degree of the dual variety $X^\vee \subset |L|$ is $c_{n-k}(J_1(L))L^{k}$.
\end{remark}

As said before, any projective variety $X \subset \P^N$ whose dual is not a hypersurface is ruled in $\P^k$'s. Let us recall the following definition:

\begin{definition}\label{def:scroll} A smooth projective embedded variety $X \subset \P^N$ is called a {\it scroll in} $\P^{n-m}$'s if $X$ admits a smooth morphism $\pi:X \to Z$ onto a smooth projective variety such that for all $z \in Z$, the fiber $\pi^{-1}(z) \subset \P^N$ is a linear $\P^{n-m} \subset \P^N$. 
\end{definition}

Observe that if $X=\P_Z(E)$ where $Z$ is a smooth variety and $E$ is a very ample vector bundle of rank $n-m$, then the embedding of $X$ in $\P^N$ is a scroll in $\P^{n-m}$'s. Scrolls are a source of examples of positive defect varieties. For further use we remark the following:

\begin{remark}\cite[Thm.~7.14]{Tev},\cite{LS}\label{ex:scrolls} If $X  \subset \P^N$ is a scroll in $\P^{n-m}$'s then $\dim X^\vee =N-1-k$ with $k \geq n-2m$.
\end{remark}

 
The first jet vector bundle of $(X,L)$ lies, by definition, in the following exact sequence ($\Omega_X$ stands for the dual of the tangent bundle to $X$):
\begin{equation}\label{eq:jetsequence}0 \to \Omega_X(L) \to J_1(L) \to \cO_X(L) \to 0
\end{equation}
 so that its top Chern class is:
\begin{equation}\label{eq:formula} c_n(J_1(L))=\sum_{k=0}^n (n+1-k)c_k(\Omega_X)L^{n-k}
\end{equation}

\begin{remark}\label{rmk:additivity} Consider $M$ another line bundle on $X$. The exact sequence (\ref{eq:jetsequence}) can be, on one hand, applied to $L+M$ and, on the other, twisted by $\cO_X(M)$ to get these two exact sequences:
$$0 \to \Omega_X(L+M) \to J_1(L+M) \to \cO_X(L+M) \to 0$$ and
$$0 \to \Omega_X(L+M) \to J_1(L)\otimes \cO_X(M)  \to \cO_X(L+M) \to 0$$

The Chern classes of the middle terms are equal:
\begin{equation}\label{eq:additivity}c_i(J_1(L+M))=c_i(J_1(L) \otimes \cO_X(M))=\sum_{t=0}^i {n+1-t \choose i-t}c_t(J_1(L))M^{i-t}\end{equation}
\end{remark}

The following proposition will be useful to study the behaviour of $c_n(J_1(L))$ when $L$ is moving in the nef cone of $X$.

\begin{proposition}\label{prop:crece} Let $X$ be a smooth projective variety of dimension $n$, $L$ and $M$ two line bundles on $X$. Suppose that $J_1(L)$ is nef (this occurs, for instance, when $L$ is very ample) and $M$ is globally generated. Then 
$$c_n(J_1(L+M)) \geq c_n(J_1(L))$$
\end{proposition}

\begin{proof} In view of Remark \ref{rmk:additivity} it is left to prove that 
$$
c_k(J_1(L))M^{n-k}\geq 0  \hbox { for any } n-1 \geq k \geq 0
$$
Since $M$ is globally generated we claim that $M^{n-k}$ can be written as a finite sum with non-negative coefficients of 
irreducible subvarieties of codimension $n-k$ in $X$. This will prove the theorem because the intersection of the Chern classes of a nef vector bundle with irreducible varieties is non-negative, see \cite[Thm.~8.2.1]{laz}. 

The claim is obiously true when $n-k=1$. Now, if the property holds for $i \geq 1$ then $M^i =\sum a_j Z_j$, $a_j \geq 0$ and $Z_j$ irreducible of codimension $i$. Since $M$ is globally generated one can choose an element $H \in |M|$ not containing any of the $Z_j$'s so that $M^{i+1}=\sum a_j (Z_j \cap H)$, as requested. 
\end{proof}

%
%


\subsection{Toric varieties}\label{sec:toric}

Let us focus on toric varieties.

\begin{notation}\label{not:toric} {\rm Assume from now on that $X$ is a smooth projective toric variety associated to a fan $\Sigma_X$, whose subset of one dimensional cones is $\Sigma(1)=\{\rho_1, \dots, \rho_m\}$ and denote by $D_i$ the divisor corresponding to $\rho_i$. Consider $|L|$ a base point free linear system on $X$. Any divisor $D \in |L|$ is linearly equivalent to a non-negative integer combination of the divisors $D_i$, see \cite[(6.4.10)]{cox}, that is, $D=\sum_i a_i D_i ,$ $a_i\geq 0$.}
\end{notation}

The generalized Euler sequence of toric varieties
\begin{equation}\label{eq:euler}0 \to \Omega_X \to \oplus_{i=1}^m \cO(-D_i) \to \cO^{m-n} \to 0\end{equation}
is showing that the total Chern class $c(\Omega_X)$ is just computed by the $D_i$'s, that is 
\begin{equation}\label{eq:chernomega}
c(\Omega_X)=\prod_{i=1}^m(1-D_i)
\end{equation}
This can be plugged in (\ref{eq:formula}) to compute the top Chern class of $J_1(L)$. 

We would like to face the problem of classification of discriminant defective toric varieties, extending the results of classification of embedded toric varieties whose dual is not a hypersurface, see \cite{dirocco}, in the following context:

\begin{problem}\label{problem} Classify $(X,L)$ such that $X$ is an irreducible smooth toric variety of dimension $n$, $L$ is a base point free line bundle on $X$ of $\dim |L|=N$, and $k=N-1-\dim(\cD)>0$, where $\cD=\cD(X,L)$ stands for the discriminant of $(X,L)$. 
\end{problem}

Moreover these varieties form a (strict) subset of those for which $c_n(J_1(L))$ vanishes. The question of its classification appears.

\begin{problem}\label{problem2} Classify $(X,L)$ as in Problem \ref{problem} such that $c_n(J_1(L))=0$.
\end{problem}

The classification of \cite{dirocco} shows that when $L$ is very ample, the only toric varieties $X \subset |L| ^\vee$ whose dual is not a hypersurface are decomposable projective bundles over a smooth toric varieties $Y$, i.e., $X =\P_{Y}(L_0 \oplus \dots \oplus L_{m})$, for $n>2m$, embedded as scrolls in $\P^m$'s. Formulae of $c_n(J_1(L))$ with $L$ nef for projective bundles like these when $Y=\P^{n-m}$ are given in Example \ref{ex:projectivebundles}.

As said before, to deal with these problems we need to understand, on one hand the behaviour of the morphism $\phi_L$, and on the other hand the geometry of the image $\phi_L(X) \subseteq |L|^\vee$.
Let us first describe how the map $\phi_L$ looks like, cf. \cite[Sect.~6.2]{cox}.
The decomposition of $D \in |L|$ as a linear combination of the $D_i$'s produces, via its Cartier Data, a lattice polytope ${P_D}$ and finally a normal generalized fan $\Sigma_{P_D}$. Then, see \cite[Thm.~6.2.8]{cox}, the fan $\Sigma$ defining $X$ is a refinement of $\Sigma_{P_D}$ and thus it induces a proper toric morphism $\phi:X \to X_{\Sigma_{P_D}}$, such that $L=\phi^*(L_{P_D})$ , where
$ X_{\Sigma_{P_D}}$ is the normal toric variety associated to $\Sigma_{P_D}$ and $L_{P_D}$ is the ample line bundle 
associated to $P_D$. The fibers of this morphism are connected, being  the associated map between their lattices surjective, cf. \cite[Prop.~2.1]{cataldo}, and then $H^0(X,L) \simeq H^0(X_{\Sigma_{P_D}}, L_{P_D})$. This is saying that $\phi_L: X \to {\P}^N$ factors through $\phi$ and the corresponding finite morphism $\phi_{L_{P_D}}$ as:
\begin{equation}\label{eq:factors}
X \stackrel{\phi}{\rightarrow}  X_{\Sigma_{P_D}} \stackrel{\phi_{L_{P_D}}}{\rightarrow} {\P}^N
\end{equation}

Consider now the dual variety $\phi_L(X)^\vee \subset |L|$, which is contained in $\cD$.  Therefore, when the defect is positive, the geometry of $\phi_L(X)$ is special, because its dual variety is not a hypersurface. For instance, it is ruled in linear spaces. It will be very useful the description of its geometry provided in \cite[Thm.~1.4]{FI} which we recall in the last statement of the following lemma. Let us summarize all these facts:

\begin{lemma}\label{lemma:phi} Let $(X,L)$ be as in Problem \ref{problem} and $\phi_L:X \to \phi_L(X) \subset |L|^\vee$ the morphism defined by the linear system $|L|$. In these conditions:
\begin{itemize}
\item[(i)] $\phi_L$ factors as described in {\rm(\ref{eq:factors})};
\item[(ii)] $\dim(\phi_L(X)^\vee)=N-1-k'$, with $k'>0$ and $k' \geq k$;
\item[(iii)] there exists a torus equivariant dominant rational map $\varphi:\phi_L(X) \to (\mathbb{C}^*)^c$ such that the closure of each fiber is projectively equivalent to the join of $r+1$ non-defective toric varieties $X_i \subset \P^{N_i}$, $0 \leq i \leq r$ such that $\P^{N_i} \cap \P^{N_j }=\emptyset$ when $i \ne j$, and $k'=r-c$. 
\end{itemize}
\end{lemma}

Besides the stated problems, it is also a natural question to classify embedded projective toric varieties whose dual variety is of low degree, trying to extend in this context the results of \cite[Thm.~5.2]{zak}, where the cases of degree of the dual variety less than or equal to three are considered. 

Since $X$ is smooth and toric, linear equivalence
and numerical equivalence for divisors coincide,  so that the real vector space $\NU(X)$ of $\R$-divisors modulo numerical equivalence coincides with $\Pic(X) \otimes \R$, and the  
Picard group of $X$ is its lattice of integer
points. Descriptions of the nef and ample cone, denoted $\Nef(X)$ and $\Amp(X)$, of a smooth projective toric variety $X$ are known, and one can denote $\Nef(X)_{\mathbb Z} = \Nef(X) \cap \Pic(X)$ and $\Amp(X)_{\mathbb Z} = \Amp(X) \cap \Pic(X)$. 
Proposition \ref{prop:crece} allows us to study the function 
\begin{equation}\label{eq:cn} \begin{array}{l}
jc_n: \Nef(X)_{\mathbb Z} \longrightarrow \Z \\
\phantom{jc_n:} \;\; L \mapsto jc_n(L)=c_n(J_1(L))
\end{array} \end{equation}
sending any nef divisor $L$ to the Chern class  $c_n(J_1(L))$ of its first jet bundle. Recall, see \cite[Thm.~6.3.20]{cox}, that the nef cone $\Nef(X)$ is rational polyhedral. This implies
the existence of a set of {\it minimal generators}, that is, nef divisors $L_1, \dots, L_m$ such that each nef divisor is a rational linear combination of the $L_i$'s with non-negative coefficients, and  each $L_i$ is primitive and spans an edge of the cone.

\begin{proposition}\label{prop:mindeg} Let $X$ be an irreducible smooth projective toric variety of dimension $n$ and let $L_1, \dots, L_m$ be a minimal generators of the nef cone $\Nef(X)$. Then the  restriction $$jc_n|_{\Amp(X)_{\mathbb Z}}: \Amp(X)_\Z\to \Z$$ of the function $jc_n$ defined above achieves its minimal value in the finite set: 
$$A=\{\sum_{i=1}^m \lambda_i L_i:  0 \leq \lambda_i \leq 1\} \cap \Amp(X)_{\Z}$$
\end{proposition}

%

\begin{proof}
As a consequence of the equivalence between ampleness and very ampleness of divisors on smooth projective toric varieties (cf. \cite[Thm.~6.1.15]{cox}) any divisor $L$ in $\Amp(X)_\Z$ is in fact very ample, so that $J_1(L)$ is globally generated and hence nef. Moreover,  the equivalence on divisors between nefness and spannedness by global sections (see \cite[Thm.~6.3.12]{cox}) says that any divisor in $\Nef(X)_\Z$ is if fact globally generated. Consider $L \in \Amp(X)_\Z \setminus A$, that can be written as $$L=p_1L_1+\dots +p_mL_m$$ with $p_i \in \Q$, $p_i \geq 0$ for $1 \leq i\leq m$. Now write 
$$L=\left \lfloor p_1 \right \rfloor L_1+\dots+\left \lfloor p_m \right \rfloor L_m +E$$
If $E$ is ample, then $E \in A$ and we conclude by Proposition \ref{prop:crece} that $c_n(J_1(L)) \geq c_n(J_1(E))$.  If $E$ is nef but not ample then there exists a curve $C$ such that $EC=0$. Hence, $L_iC>0$ implies that the coefficient of $L_i$ in $E$ is $0$. But, since $L$ is ample there exists $1 \leq i \leq m$ such that $L_iC>0$ and $\left \lfloor p_i \right \rfloor \ne0$, otherwise $LC=0$. Hence, we can substitute $E$ by $E+L_i$ and repeat the process in case $E+L_i$ is nef but not ample, leading finally to a choice of $E$ ample. 
\end{proof}

Proposition \ref{prop:crece} reads particularly simple in the following situation:

\begin{corollary}\label{cor:mindeg} In the conditions above, if there exists an ample divisor $L$ on $X$ such that any ample divisor $L'$ can be written as $L'=L+D$ with $D$ a nef divisor, then $c_n(J_1(L)) \leq c_n(J_1(L'))$, that is, the minimal value of $jc_n$ is reached at $L$.
\end{corollary}


%
%
%
In relation to our problems of classification of positive defect varieties or varieties whose dual has low degree, one can study simultaneously all different line bundles over a chosen $X$. 

\begin{corollary}\label{cor:mindeg2} In the conditions of Proposition \ref{prop:mindeg}, if $$\min\{c_n(J_1(L)): L \in A\} \neq 0$$ and it is obtained in, say, $L_{min}$, then for any embedding $X \subset \P^N$, the following holds:
\begin{enumerate} 
\item[(i)] the dual variety $X^\vee \subset\P^{N\vee}$ is a hypersurface, and
\item[(ii)] the degree of the dual variety $\deg(X^\vee)$  is greater than or equal to the degree of  the dual variety corresponding to the linearly normal embedding $X \subset |L_{min}|^\vee$.
\end{enumerate}
\end{corollary}

\begin{proof} The proof is just an application of Proposition \ref{prop:mindeg} and of the fact that the dual variety of a linear isomorphic projection of a projective variety $X \subset \P^N$ is the corresponding linear section of $X^\vee \subset \P^{N\vee}$, see \cite[Thm.~1.21]{Tev}.
\end{proof}

\section{A bunch of examples}\label{sec:examples}

The following examples are of interest:

\begin{example}[Projective space]\label{ex:projective} Let $X=\P^n$ be the projective space, easy computations lead to \begin{equation}\label{eq:proj} c_n(J_1(\cO(m)))=(n+1)(m-1)^n \end{equation}
It only vanishes when $m=1$, for which the discriminant is empty. 

\end{example}

\begin{example}[Projective bundles over projective spaces]\label{ex:projectivebundles} Let $\Sigma$ be the fan defining the projective bundle $X=\P_{\P^{n-m}}(\cO \oplus \cO(r_1) \oplus \dots \oplus \cO(r_m))$, where $0 \leq r_1\leq \dots\leq r_m$. Let $s=\min(\{m+1\}\cup\{i: r_i\neq 0\})$. The fan $\Sigma \subset \mathbb{R}^{n}=\R^{n-m} \times \R^m$ is defined by $\Sigma(1)=\{\rho_0, \dots, \rho_{n+1}\}$, where the $\rho_i$'s are the following rays (see for instance \cite[Ex.~7.3.5]{cox}): 
$$\begin{cases}
\rho_{0}= \langle ((1, \dots, 0), (0, \dots ,0)) \rangle \\
\dots \\
\rho_{n-m-1}= \langle ((0, \dots, 1), (0, \dots ,0)) \rangle \\
\rho_{n-m}= \langle((-1, \cdots, -1),(r_1, \dots,r_m)) \rangle \\ 
\rho_{n-m+1}=\langle((0, \dots, 0),(1 , \dots, 0))\rangle \\
\dots \\
\rho_{n}= \langle((0,\dots,0),(0, \dots, 1))\rangle\\
\rho_{n+1}=\langle((0, \dots, 0),(-1,\dots, -1))\rangle
\end{cases}
$$



Denoting by $D_i$ the divisor associated to $\rho_i$, the following relations in the Picard group appear:
$$
D_0= \dots =D_{n-m}, \quad
D_{n+1}=D_{n-m+j}+r_jD_0
$$
and the nef cone of $X$ is generated by $D_0$ and $D_{n+1}$. Moreover, any nef line bundle can be written as $L=aD_0+bD_{n+1}$ for $a,b$ non-negative integers. The Formula (\ref{eq:chernomega}) shows:
$$c_l(\Omega_X)=(-1)^l\sum_{0 \leq \i_1<\dots<\i_l\leq n+1}D_{i_1}\dots D_{i_l},$$
and in view of formula (\ref{eq:formula}) we get:
\begin{equation}\label{eq:projectivebundles}
c_n(J_1(L))=\sum _{l=1}^n (-1)^l (n+1-l)(aD_0+bD_{n+1})^{n-l}( \sum_{0 \leq \i_1<\dots<\i_l\leq n+1}D_{i_1}\dots D_{i_l})
\end{equation}

When $b=0$, since $D_0=\pi^*(\cO_{\mathbb{P}^{n-m}}(1))$, where $\pi:X \to \mathbb{P}^{n-m}$ is the morphism defining the projective bundle structure, then the morphism defined by $|L|=|aD_0|$ is just the composition of $\pi$ with the $a$-Veronese embedding of $\mathbb{P}^{n-m}$.

When $a=0$, $|L|=|bD_{n+1}|$ defines the $b$-Veronese embedding of the cone whose vertex is a linear $\P^{s-1}$ over the embedding defined by $\cO(1)$ of $\P(\cO(r_{s+1}) \oplus \dots \oplus \cO(r_m))$. 

When $L$ is ample, i.e. $a,b>0$, then it is very ample and the minimal value of $c_n(J_1(L))$ is reached when $a=b=1$ in formula (\ref{eq:projectivebundles}) by Corollary \ref{cor:mindeg}.
\end{example}

Let us consider in detail some particular cases.\medskip

\noindent$\bullet$ {\it Contraction to the base.} When $L=aD_0$, i.e. $b=0$, we get that 
$$c_n(J_1(aD_0))=\sum_{l=0}^n(-1)^l(n+1-l)(m+1)a^{n-l}{n-m+1 \choose l-m}=$$
$$=(-1)^m(m+1)(n-m+1)(a-1)^{n-m}$$  As $|L|=|D_0|$ gives the morphism defining the projective bundle structure, its discriminant is empty, and $c_n(J_1(D_0))=0$. For $L=aD_0$, $a>1$, the morphism defined by $|L|$ is just the composition of $\pi$ with the $a$-Veronese embedding of $\mathbb{P}^{n-m}$ so that $c_n(J_1(L))<0$  does not contradict Lemma \ref{lem:chern}, because the singular locus of any singular element in $|aD_0|$ is of positive dimension. Moreover, as Example \ref{ex:projective} shows, the degree of the dual variety of this Veronese variety $v_a(\P^{n-m}) \subset \P^M$ is $(n-m+1)(a-1)^{n-1}$, which in fact appears as a factor in $c_n(J_1(L))$.

These examples are showing that:

\begin{remark}\label{rem:reciprocalchern1} There exist pairs $(X,L)$ as in Lemma \ref{lem:chern} such that 
$c_n(J_1(L))<0$, so that the general singular element of $|L|$ is singular along a positive dimensional subvariety.
\end{remark}

\noindent$\bullet$ {\it Products.} In the particular case $r_1=\dots = r_m=0$ the variety $X=\P^{n-m}\times \P^m$ is just the product of two linear spaces. Since the only non-zero intersection is $D_0^{n-m}D_{n+1}^m=1$ one gets  $$c_n(J_1(L))=\sum_{l=0}^n(-1)^l(n+1-l)\sum_{j=0}^{n-l}{n-l \choose j}{n-m+1 \choose n-m-j}{m+1 \choose l-n+m+j}a^j b^{n-l-j}$$ 
 Without loss of generality we can assume $n-m \geq m$. Recall that for $a=1$, $b>0$ the linear system $L=D_0+bD_{n+1}$ is providing an embedding of $X$ as a scroll in $\P^{n-m}$'s, hence the defect is greater than or equal to $n-2m$ (see Remark \ref{ex:scrolls}). In particular $c_n(J_1(D_0+bD_{n+1}))=0$ when $n>2m$ and it is giving the degree of the discriminant when $n=2m$. In this last situation we get that the minimal value is reached at $b=1$:
\begin{equation}\label{eq:products}c_{2m}(J_1(D_0+D_{n+1}))=$$ $$=\sum_{l=0}^{2m}(-1)^l(2m+1-l)\sum_{j=m-l}^{2m-l}{2m-l \choose j}{m+1 \choose m-j}{m+1 \choose l+j-m}=m+1\end{equation} where the second equality can be proved by induction and the usual properties of the combinatorial numbers.

These examples show that any possible value of $c_n(J_1(L))>1$ can be reached.
\medskip

\noindent$\bullet$ {\it Rank one.} Let us consider the case $m=1$.
$$c_l(\Omega_X)=\big({n \choose l}-r_1{n \choose l-1}\big)D_0^l+2{n \choose l-1}D_0^{l-1}D_{n+1}$$
%
%
Hence, in order to use Formula (\ref{eq:formula}), we need to compute $L^{n-l}D_0^l$ for $0 \leq l \leq n-1$ and $L^{n-l}D_0^{l-1}D_{n+1}$ for $1 \leq l\leq n-1$. For the former we get:
$$L^{n-l}D_0^l=\begin{cases}
\frac{1}{r_1}((a+br_1)^{n-l}-a^{n-l}) \quad \hbox{if } r_1\ne 0 \\
(n-l)a^{n-l-1}b \quad \hbox{if } r_1=0 
\end{cases}$$ 
and for the latter:
$$
L^{n-l}D_0^{l-1} D_{n+1}=(a+br_1)^{n-l} 
$$

Finally, using the following equalities of combinatorial numbers:
$$(n+1-l){n \choose l}={n \choose l}+n{n-1 \choose l}, \; (n+1-l){n \choose l-1}=n{n-1 \choose l-1},$$ 
together with Formula (\ref{eq:formula}) one gets the following expressions for $c_n(J_1(L))$:

\noindent For $r_1=0$, where $X$ is just the product $\mathbb{P}^{n-1} \times \mathbb{P}^{1}$:
\begin{equation}\label{eq:rankoneproducts}
c_n(J_1(L))=n(a-1)^{n-2}(2(a-1)(b-1)+ab(n-1))
\end{equation}

\noindent For $r_1>0$:
\begin{equation}\label{eq:rankone}
c_n(J_1(L))=\frac{1}{r_1}\big((a-1+br_1)^{n-1}(a-1+br_1+an+br_1n-nr_1)-
\end{equation}$$
-(a-1)^{n-1}(a-1+an+nr_1)\big)
$$

-When  $a=0$,  $b>0$, and $r_1=0$ we get $c_n(J_1(L))=-2n(-1)^{n-2}(b-1) $, vanishing only when $b=1$. Recall that here the morphism defined by $|L|=|bD_{n+1}|$ is the projection to $\mathbb{P}^1$ composed with its $b$-Veronese embedding. 

-When $a=0$, $b>0$,  and $r_1>0$ the morphism defined by $|L|=|D_{n+1}|$ is the contraction of the minimal section, for which the image is a cone of vertex a point over a $r$-Veronese embedding of $\mathbb{P}^{n-1}$ and $|bD_{n+1}|$ is defining a Veronese embedding of this cone. Since any section of $|L|$ which corresponds to a hyperplane through the vertex is reducible (contains the minimal section) then $\mathcal{ D}(X,L)$ is a hypersurface and  Lemma \ref{lem:chern} does not apply. The formula for the Chern class is:
$$c_n(J_1(L))=\frac{1}{r_1}((-1+br_1)^n+(br_1-1)(b-1)nr_1-(-1)^{n-1}(-1+r_1n))$$ which can be shown to be bigger than 0 except at $r_1=3=n$, where it vanishes. 

This example is showing that one cannot expect to have a reciprocal of Lemma \ref{lem:chern}.

\begin{remark}\label{rem:reciprocalchern2} There exist pairs $(X,L)$ as in Lemma \ref{lem:chern} such that 
 $c_n(J_1(L))=0$ and $\cD \subset |L|$ is a hypersurface. 
\end{remark}


In the rest of the cases, i.\ e.,  $a>0$, $b>0$, the line bundle $L$ is very ample so that $c_n(J_1(L))$ has to be bigger than or equal to zero, being the degree of $X^\vee \subset |L|$ when it is a hypersurface and zero when its codimension is bigger than one. Moreover, in view of Corollary \ref{cor:mindeg} the minimal value of $c_n(J_1(L))$ is reached for $a=b=1$.

-When $r_1>0$, $a=b=1$ leads to $c_n(J_1(L))=r_1^{n-2}(r_1+n)$. The minimal value is $3$ when $n=2$, $r_1=1$: \begin{equation}\label{eq:blowup}c_2(J_1(D_0+D_3))=3 \end{equation}
This corresponds to the embedding in $\P^4$ of the blowup of the plane at a point. Moreover, the equalities $c_2(J_1(D_0+D_3))= r_1+2$ for $n=2$, and $c_n(J_1(D_0+D_{n+1}))=1+n$ for $r_1=1$ provide another family of examples reaching any possible value of $c_n(J_1(L))>1$.

-When $r_1=0$, $a=1$, $b>0$ and $n=2$,  we get $c_n(J_1(L))=2b$. In this case $c_n(J_1(L))$ is the degree of the dual variety of the Segre embedding of $\P^1 \times C_b$, being $C_b \subset \P^b$ the rational normal curve of degree $b$. For $n>2$ we get $c_n(J_1(L))=0$ because $|L|$ is defining an embedding of $\P^{n-1} \times \P^1 \subset |L|^\vee$  as a scroll in $\P^{n-1}$'s, for which the dual variety is not a hypersurface. In fact the degree of the dual variety is computed as $$c_2(J_1(L))L^{n-2}=bn$$  In the particular case $n=3$, $b=1$ one gets \begin{equation}\label{eq:segre} c_2(J_1(D_0+D_4))=3 \end{equation}

When $r_1=0$, $a>1$, $b>0$,  one can check that the minimal is reached for $n=2,a=2,b=1$, which is $4$, and for the rest of the cases one gets $c_n(J_1(L))>4$. 

%
%
%
%





\medskip

\noindent$\bullet$ {\it Rank two projective bundles of dimension three.} Let us consider the case $n=3$, $m=2$. We get that 
\begin{equation}\label{eq:ranktodim3}c_3(J_1(L))=2(b -1)(2(r_1+r_2)b (b-1/2)+3(a-1)(b-1)+3ab)
\end{equation}
Since the case $r_1=r_2=0$ has been considered previously, assume $r_2>0$.

-When $b =1$, we get $c_3(J_1(L))=0$. If $a=0=r_1$ then $|D_4|$ defines the contraction of the family of minimal sections to a cone of vertex a line over the $r_2$-Veronese embedding of $\P^1$, the discriminant here is not a hypersurface. If $a=0$ and $r_1,r_2>0$ then $|D_4|$ defines the contraction of the minimal section to a cone of vertex a point over a rational ruled surface (whose fibers are linearly embedded) and the discriminant is not a hypersurface. If $a \geq 1$ then $|L|$ is very ample and $X \subset |L|^\vee$  is a scroll in planes, hence the discriminant is not a hypersurface. In this case, the degree of the dual variety is given by 
\begin{equation}\label{eq:ranktodim3def}c_2(J_1(L)) L=3a+r_1+r_2
\end{equation}

-When $b >1$,  $a=0$, $c_3(J_1(L))=2(b-1)(2(r_1+r_2)b(b-1/2)-3(b-1))>0$.

-When $b >1$, $a>0$ the line bundle $L$ gives an embedding and $c_3(J_1(L))>0$ is computing the degree of the dual variety of the corresponding embedded $3$-fold. One can check that the minimal degree is computed in $b=2$, $a=1$, hence, the extremal case is $c_3(J_1(L))=2(6(r_1+r_2)+6)$, strict inequality in the rest of the cases.

\begin{example}\label{ex:blowup} If we consider the case $m=1$ in the previous example and refine the cone  by means of the ray $\rho_{n+2}=(1, 0, \dots,0, -1)$, the corresponding toric variety $X$ is the blow-up of the projective bundle in a hyperplane of the minimal section.  The line bundle $L=D_1$ defines a morphism which contracts the exceptional divisor and then projects the projective bundle onto the base. The discriminant $\cD(X,L)$ is then a point. Let us show the value of the top Chern class of its first jet bundle. Recall that we have now that $D_1= \dots =D_{n-1}=D_0+D_{n+2}$ and $D_n=D_{n+1}+D_{n+2}+r_1D_1$. The only non-zero intersection numbers involving $D_1$ are  the following: $$1=D_1^{n-1}D_{n+1}=D_1^{n-1}D_{n}=D_0D_1^{n-2}D_n=D_0D_1^{n-2}D_{n+1}=D_0D_1^{n-2}D_{n+2}$$ Hence $$(-1)^lc_l(\Omega_X)D_1^{n-l}=3{n-1 \choose l-2}+2{n-1 \choose l-1}$$ Plugging this in formula (\ref{eq:formula}) and using basic properties of combinatorial numbers one gets: $$c_n(J_1(L))=(n-1)(3 \sum_{k=1}^{n}(-1)^k {n-2 \choose k-2} +2\sum_{k=1}^{n}(-1)^k {n-2 \choose k-1})$$ which is one when $n=2$ and vanishes when $n>2$.
\end{example} 

\begin{example}[Rank one projective bundles over Hirzebruch Surfaces]\label{ex:lbhirz} Consider $X=\P_{\P^1}(\cO \oplus \cO(r))$ defined, as above,  by the fan with rays $\Sigma_X(1) = \{\rho'_0=(1,0), \rho'_1=(-1,r), \rho'_2=(0,1), \rho'_3=(0,-1)\}$.  The nef cone of $X$ is generated by $D'_0$ and $D'_3$, and so any decomposable rank one projective bundle $Y$ over $X$ can be expressed as either $\P_X(\cO\oplus\cO(sD'_0+tD'_3))$ with $s,t \geq 0$,  or $\P_X(\cO(sD'_0)\oplus\cO(tD'_3))$ for some $s,t>0$. 

In any case the toric variety has rays given by $$\Sigma_Y(1)= \{ \rho_0,\rho_1,\rho_2,\rho_3,\rho_4,\rho_5 \}$$ where
$$\begin{cases}
\rho_0=\begin{cases}
\langle(1,0,s)\rangle \; \; \: \hbox{  if $X=\P_X(\cO \oplus \cO(sD_0'+tD_3'))$, $s,t \geq 0$}\\
\langle(1,0,-s)\rangle \;  \hbox{  if $X=\P_X(\cO(sD_0') \oplus \cO(tD_3'))$, $s,t >0$} \end{cases}\\
\rho_1 = \langle(-1,r,0)\rangle, \\
\rho_2=\langle(0,1,0)\rangle,\\
\rho_3=\langle(0,-1,t)\rangle,\\
\rho_4=\langle(0,0,1)\rangle,\\
\rho_5=\langle(0,0,-1)\rangle \\
\end{cases}$$
which corresponds to the toric divisors $D_0$, $D_1$, $D_2$, $D_3$, $D_4$ and $D_5$. Moreover, any nef line bundle can be written in the form $L=aD_0+bD_3+cD_5$ in the case $Y_0= \P_X(\cO\oplus\cO(sD'_0+tD'_3))$, and $L=aD_0+bD_3+c(sD_0+D_5)$ in the case $Y_1= \P_X(\cO(sD'_0)\oplus\cO(tD'_3))$ for some integers $a,b,c\geq0$. 

In this setup we obtain the following result.

\begin{proposition}
Let $X=\P_{\P^1}(\cO\oplus\cO(r))$ be a Hirzebruch surface with $r\geq0$, let $Y$ be the projectivization of a decomposable rank two vector bundle on $X$ and let $L$ be a nef line bundle on $Y$. Then $c_3(J_1(L))=0$ if and only if either $L=D_0$, or $r=0$ and $L=D_3$, or $s=t=0$ and $L=D_5$; or $(Y,L)$ is one of the following cases:
\[
\begin{array}{lll}
X & Y & L \\[0.05cm]
\hline\\[-0.3cm]
(a) \;\P_{\P^1}(\cO\oplus\cO(r)) 
		& \P_X(\cO\oplus\cO(sD'_0+D'_3)),  ~s\geq0   & D_0+D_5 \\[0.1cm]
(b)		& \P_X(\cO(D'_0)\oplus\cO(D'_3)),    & D_0+D_5 \\[0.1cm]
\hline\\[-0.3cm]
(c)\; \P^1\times\P^1 & \P_X(\cO\oplus\cO(D'_0+tD'_3)), ~t\geq0	& D_3+D_5 \\[0.1cm]
(d)			& \P_X(\cO\oplus\cO(2D'_0+2D'_3))	& D_5 \\[0.1cm]	
\hline\\[-0.3cm]
(e) \: \P_{\P^1}(\cO\oplus\cO(1))	& X\times\P^1	 & D_3+D_5 \\[0.1cm]
(f)	& \P_X(\cO\oplus\cO(D'_0+2D'_3))	& D_5 \\[0.1cm]
(g)	& \P_X(\cO\oplus\cO(3D'_3))		& D_5 \\[0.1cm]
\hline\\[-0.3cm]
(h)\; \P_{\P^1}(\cO\oplus\cO(2))	& \P_X(\cO\oplus\cO(2D'_3))	& D_5\\[0.1cm]
\hline
\end{array}
\]

\end{proposition}

\begin{proof} The list of cases shown above is obtained by inspection of the explicit formulae of $c_3(J_1(L))$. First observe that in both cases, $\P_X(\cO\oplus\cO(sD'_1+tD'_2))$ and $\P_X(\cO(sD'_1)\oplus\cO(tD'_2))$, every nef divisor is a non-negative integer combination of the generators of the nef cone so by Corollary \ref{cor:mindeg} the minimal value for $c_3(J_1(L))$ for $L$ ample is obtained when $(a,b,c)=(1,1,1)$. 

Let us start with the case $Y_0=\P_X(\cO\oplus\cO(sD'_1+tD'_2))$, where we have that 
\begin{equation}\label{eq:projhirz}c_3(J_1(L)) = 4c^3rt^2 + 12bc^2rt + 8c^3st - 3c^2rt^2 + 12b^2cr + 12bc^2s + 12ac^2t -\end{equation}
$$- 6bcrt- 3c^2rt - 6c^2st + 24abc - 6b^2r - 6bcr - 6bcs - 6c^2s - 6act -$$
$$- 6c^2t  
+ 2crt - 12ab - 12ac - 12bc + 4br + 4cs + 4ct + 8a + 8b + 8c - 8$$

Then the minimum is $rt^2 + 5rt + 2st + 4r + 4s + 4t + 4$ which is always greater than zero. Then, it is left to check the vanishing of  $c_3(J_1(L))$ at the walls $a=0$, $b=0$ and $c=0$ of the nef cone.

If $a=0$ then $c_3(J_1(L)) = (3c^2rt^2 + 6bcrt + 6b^2r + 6bcs + 8)(c-1) + (3c^2rt + 6bcr + 6c^2s)(b-1) + c^2st(8c-6)+ c^2t(crt-6) + 3bc(crt-4) + 2crt + 4br + 4cs + 4ct + 8b$ which may only be less than or equal to zero for $c=0$, $c=1$ or $b=0$. If $c=0$ then $b=1$ and we obtain the case $r=0$ and $L=D_3$. If $c=1$ then the only possible value for $b$ is $1$, so $L=D_3+D_5$, and the values for $r$ are only $0$ or $1$. In the first one we obtain the case $(c)$ and in the second the case $(e)$ in the table. Finally, if $b=0$ then we must have $c=1$, so $L=D_5$. Now if $r,s,t>0$ then the only possibility happens to be $r=s=1$ and $t=2$ obtaining the case $(f)$ in the table. If $r=0$ but $s,t>0$ then $s=t=2$ so we get the case $(d)$. If now $s=0$ but $r,t>0$ then either $r=t=2$ getting the case $(h)$ or $r=1,t=3$ getting the case $(g)$. The remaining case for the vanishing of $c_3(J_1(L))$ is obtained when $s=t=0$ for any value of $r$. 

If $b=0$ (and $a>0$) then $c_3(J_1(L)) = (3c^2rt^2 + 6c^2st + 8)(c - 1) + (c^2rt + 2c^2s + 4ac)(ct - 3) + 2act(2c - 3) + 2c^2t(2a - 3) + 2crt + 4cs + 4ct + 8a$, so only for $c=0$ and $c=1$ we may obtain the value zero. In the first case we must have $a=1$, so $L=D_0$ for any value of $r$, $s$ and $t$ gives $c_3(J_1(L)) =0$. For the case $c=1$ we must have $t=1$ but there is no restriction on $r$ and $s$ so we obtain the case $ (a)$ in the table. 

Finally if $c=0$ (and $a,b>0$) then $c_3(J_1(L)) = - 2br(3b - 2) - 2a(3b - 4) - 2b(3a - 4) - 8$ which is always smaller than zero.

For the case $\P_X(\cO(sD'_1)\oplus\cO(tD'_2))$ we have that 
\begin{equation}\label{eq:projhirz2}c_3(J_1(L)) = 4c^3rt^2 + 12bc^2rt + 4c^3st - 3c^2rt^2 + 12b^2cr + 12bc^2s + 12ac^2t-
\end{equation} 
$$- 6bcrt - 3c^2rt + 24abc - 6b^2r - 6bcr - 6bcs - 6c^2s - 6act - 6c^2t +$$
$$+ 2crt - 12ab - 12ac - 12bc + 4br + 4cs + 4ct + 8a + 8b + 8c - 8 $$
and following similar calculations as in the previous paragraphs, we obtain the remaining case $(b)$.
\end{proof}

\end{example}

Observe that in cases $(a)$, $(c)$ and $(e)$ the image $\phi_L(X) \subset |L|^\vee$ is a smooth scroll in planes, so that its dual variety is not a hypersurface. There are extra components in the discriminant. In case $(b)$, on the section corresponding to the first summand, we get a map onto a (linear) $\P^1$, and on the second one we get a map onto a cone, which is the contraction of the minimal section of $X$. Therefore $\phi_L(X)$ is ruled in planes parametrized by a $\P^1$.
In the rest of the cases the image $\phi_L(X) \subset |L|^\vee$ is a cone with vertex a point, hence the discriminant contains a hyperplane, and the singular locus of the general element in the discriminant is positive dimensional.

\section{Toric varieties with big defect}\label{sec:toricdisc}

 The defect is known to be bounded by the dimension of the variety and a complete classification of the extremal cases can be applied to the case of toric varieties. Moreover, similar bounds on the defect can be stated in terms of the dimension on the linear system $|L|$. The following results are written in \cite{LM1}, let us collect them conveniently for our purposes. 

\begin{lemma}\label{lemma:lm1summary} Let $(X,L)$ be as in  Problem \ref{problem} and let $f$ be the maximal dimension of a fiber $F$ of $\phi_L$. Then $k \leq $ {\em min}$\{n-f,N\}$. Moreover, if $k=n-f$ then, for any $x,x' \in F$  it holds  that $|L-2x| =|L-2x'|$ is a linear component of maximal dimension of the discriminant $\cD(X,L)$.
\end{lemma}

\begin{proof} Since $\dim\cD(X,L) \geq -1$ then $k \leq N$. For the other inequality observe first that if $f=0$ then $L$ is ample, so that very ample, and $\phi_L$ is an embedding. Hence $k \leq n$, because its is the dimension of the general contact locus, which is a subvariety of $X$. If $f>0$ we have that that $|L-x|=|L-F|$ for any $x \in F$. Hence, the codimension of $|L-2x|$ in $|L-x|$ is smaller than or equal to $n-f$, and then $\dim |L-2x| \geq N-1-(n-f)$. As $|L-2x| \subset \cD(X,L)$ we get the expected inequality. If equality of dimensions holds, that is, if $k=n-f$, then $|L-2x|$ is a linear component of maximal dimension of $\cD(X,L)$ and moreover, $|L-2x|=|L-2x'|$ for $x' \in F$, as stated. 
\end{proof}

Let us comment the extremal cases.

\begin{lemma}\label{lema:bound} With the notation of Lemma \ref{lemma:lm1summary} we get that $k \leq$ {\em min}$\{n,N\}$, and moreover:
\begin{itemize}

\item[(i)] If $k=n$ then $(X,L)=(\P^n, \cO(1))$.
\item[(ii)] If $k=n-1$ then $X=\P_{\P^{n-1}}(\cO \oplus \cO(r_1))$ is a toric projective bundle of rank one $\pi:X \to \P^{n-1}$ and $L=\pi^*(\cO(1))$.
\item[(iii)] If $k=N$ then $\phi_L:X \to \P^{N}$ is a smooth toric surjective morphism of connected fibers; 
\item[(iv)] If $k=N-1$ then $\phi_L:X \to \P^{N}$ is surjective. 
\end{itemize}
\end{lemma}

%

\begin{proof} Statement (i) is just \cite[Thm~1.5]{LM1}. To prove (ii) we first consider $n=2$ so that by \cite[Thm.~2.2]{LM1} either $X$ is as in case (ii), or $X=C \times \mathbb{P}^1$, genus of $C$ greater than or equal to one, which does not fit with $X$ toric. For $n \geq 3$ we use similarly \cite[Thm.~3.2, Thm.~4.1]{LM1}. For (iii) just observe that $\dim \cD=-1$ says that $\cD$ is empty and \cite[Prop.~1.4]{LM1} applies, hence $\phi_L:X \to \P^{k}$ is a smooth morphism; we conclude by Lemma \ref{lemma:phi}(i).  In the last statement (iv) use  Lemma \ref{lemma:phi}(ii) to get $k' \geq k=N-1$ ($\phi_L(X)^\vee$ is either empty or zero dimensional), which implies $\dim (\phi_L( X)) = N$,  not being linear $\phi_L(X) \subset |L|^\vee$ if strictly contained.
%
%
\end{proof}


\begin{remark}\label{rem:(4)} Locally trivial toric fibrations over $\P^N$ provide examples of (iii), for instance projective bundles where, with the notation of Example \ref{ex:projectivebundles}, $L=D_0$. An example of (iv) is presented in Example \ref{ex:blowup}. In fact, this is a general way to produce pairs $(X,L)$ with linear discriminant. Indeed, consider $(X,L)$ as in (iii) with fibers of positive dimension, $N<n$, and suppose that $\phi_L$ has sections. One can take the image of a linear space $T \subset \P^k$  in the target via one of the sections of $\phi_L$. The blow up of $X$ along this subvariety has a map to $\P^k$ defined by a linear system whose only singular elements correspond to the linear space of hyperplane sections of $\P^k$ passing through $T$.
\end{remark}

A posible next situation to consider is $k=n-2 >0$.
\medskip

Let us write this lemma that we will need for this case and later on. For any projective variety $Y \subset {\mathbb P^N}$ we denote by $\langle Y \rangle$ the minimal linear space in ${\mathbb P^N}$ containing $Y$.

\begin{lemma}\label{lemma:coneoveracurve} Let  $(X,L)$ be as in Problem \ref{problem} such that $\phi_L(X) \subset |L|^\vee$ is a $m$-dimensional cone over a non-defective curve $C$ with vertex $V=\mathbb{P}^{m-2}$.
 If $\langle C \rangle \cap V=\emptyset$ then $\dim(X,L)\geq \dim |L|-1-(m-2)$. 
\end{lemma}

\begin{proof}  Since $\langle C \rangle \cap V=\emptyset$ then there exists a rational map $r: X \to C$ defined as $r(x)=\langle V,x\rangle \cap C$. Consider now a general hyperplane $H$ containing $V$, which is meeting $C$ in a finite set of points $p_1, \dots, p_d$, $d>1$ (recall that $C$ is not a line). Then the corresponding element in $|L|$, called $H$ by abuse of notation, contains the union of the divisors $r^{-1}(p_i)$ and consequently it is reducible. Since $\dim \phi_L(X)>1$, $H$ is connected by Bertini Theorem and finally singular.
\end{proof}

Hence, when $k=n-2 >0$ we get the following:

\begin{proposition}\label{prop:notbig}  Let  $(X,L)$ be as in Problem \ref{problem},  $k=n-2>0$. If $L$ is not big then $(X,L)$ is as described in {\em (iii)} and {\em (iv)} of Lemma \ref{lema:bound}.
%
\end{proposition}

\begin{proof} By Lemma \ref{lema:bound} we can assume that $k=n-2 \leq N-2$, that is $n\leq N$.
Then, $\phi_L(X) \subset |L|^\vee$ is a positive defect variety whose defect $k'$ satisfies $k' \geq n-2>0$. Moreover, since $\phi_L(X) \subset |L|^\vee$ is not linear then $\dim \phi_L(X) \geq n-1$, in fact equal, not being $L$ big, and so $k'=n-2$. By  Lemma \ref{lemma:phi}(iii), $\phi_L(X) \subset |L|^\vee$ is a cone over a non-defective toric curve $C$ with vertex $V$, a linear $\mathbb{P}^{n-3}$. This contradicts $k =n-2$ by Lemma \ref{lemma:coneoveracurve}.

\end{proof}


Let us propose the following definition

\begin{definition}\label{def:genscroll} A smooth projective variety $X$ is called a {\em generalized scroll over a line} if there exists a pencil $P$ of divisors and a base point free divisor $L$  such that $\dim |L| \geq \dim X$ and $\phi_L:X \to |L|^\vee$  is sending the general element $D \in P$ onto a linear space of dimension $n-1$ inside $|L|^\vee$. 
\end{definition}

Observe that, if $X$ is a generalized scroll over a line, then the image $\phi_L(X) \subset |L|^\vee$ is ruled in $\P^{n-1}$'s. The projective space $\P^n$, $n \geq 2$, is a generalized scroll over a line just taking the pencil $P$ of hyperplanes contained a fixed $\P^{n-2}$ and $L=\cO(1)$.

\begin{example}\label{ex:genscrolll} 
This generalizes the notion of scroll in $\P^{n-1}$'s over $\P^1$: if there is such structure of scroll, say $\pi: X \to \P^1$, embedded by $|L|$, then one can choose $P=|\pi^*\cO(1)|$.

\end{example}

When $k=n-2$ one expects to have linear spaces of dimension $n-2$ through a general point of $\phi_L(X)$ but, as in the very ample case, one gets linear spaces of bigger dimension. 

\begin{theorem}\label{thm:big} Let  $(X,L)$ be as in Problem \ref{problem},  $k=n-2$. If $L$ is big then 
 $(X,L)$ is a generalized scroll over a line.
\end{theorem}

\begin{proof} Since $L$ is big, then $\dim \phi_L(X)=n$. If $\dim |L|=n$ then $\phi_L(X)=\P^n$ and we have a structure of generalized scroll over a line. Assume in the sequel that $\dim |L|>n$. Thus, $k' \geq n-2$, see Lemma  \ref{lemma:phi}(ii). 

If $k'=n-1$ then, by Lemma \ref{lemma:phi}(iii), $c=0$ and $\phi_L(X)$ is a cone over a non-defective curve $C$ (rational,  being toric), whose vertex $V$ is a linear $\P^{n-2}$. This provides a structure of generalized scroll over a line. In fact, consider the rational map $r: \phi_L(X) \to C$ defined on a general point $x \in \phi_L(X)$ as $r(x)=\langle V,x \rangle \cap C$ and the closure of its general fibers are linear spaces of dimension $n-1$. 


If $k'=n-2$ then, again by Lemma \ref{lemma:phi}(iii), either $c=1$ or $c=0$. If $c=1$  then there exists a rational map $\varphi: X \to \mathbb{C}^*$ whose fibers are linear $\mathbb{P}^{n-1}$'s. This gives the structure of generalized scroll over a line. If $c=0$ then $\phi_L(X)$ is of one of these types:


Case I: a cone over a join of two toric non-defective curves, say $C_1$ and $C_2$ with vertex $V$ of dimension $n-4$. 

Case II: a cone over a toric non-defective surface $S$ which vertex $V$ of dimension $n-3$.

Let us consider Case I. Since $C_2$ is not a line and $V \cap \langle C_2\rangle =\emptyset$ then $$\dim \langle V, C_2 \rangle \geq n-4+2+1=n-1.$$ If equality holds and the general $p_1 \in C_1$ belongs to $\langle V, C_2 \rangle$ then $\phi_L(X) \subset \P^{n-1 }$, a contradiction, hence for the general $p_1 \in C$, and, by the same argument, for the general $p_2 \in C_2$, one gets:
\begin{equation}\label{eq:joinbound} \dim \langle V,p_1,C_2  \rangle \geq  n, \; \dim \langle V,p_2,C_1  \rangle \geq  n,
\end{equation}
Consider the cone $Y \subset \phi_L(X)$ over $C_2$ with vertex in the linear space spanned by $V$ and a general $p_1 \in C_1$. Let $Y_0 \subset Y$ be the open subset of points in  $Y$ which are not singular in $\phi_L(X)$. Denote by $\cD_{p_1} \subset \phi_L(X)^\vee \subset \cD$ the closure of the subvariety of hyperplanes in $\phi_L(X)^\vee$ singular at points in $Y_0$.
Since $\dim \phi_L(X)^\vee =N-1-(n-2)$ then 
\begin{equation}\label{eq:dimDp1}\dim \cD_{p_1}  =\dim \phi_L(X)^\vee-1=N-n
\end{equation}
 Since $\phi_L(X)$ is a cone over the join of $C_1$ and $C_2$ with vertex $V$, then for a general $y \in Y$ there exists $p_2 \in C_2$ such that $y \in \langle V,p_1,p_2\rangle$ and, by Terracini Lemma, the tangent space $T_{\phi_L(X),y}$ to $\phi_L(X)$ at $y$  is the linear span
\begin{equation}\label{eq:terracini}
T_{\phi_L(X),y}=\langle V, T_{C_1,p_1}, T_{C_2,p_2} \rangle. 
\end{equation}
 Therefore, any hyperplane $H$ tangent to $\phi_L(X)$ at $y \in Y_0$ is singular along the linear space $\langle V, p_1, p_2\rangle$. This implies that the corresponding element in $|L|$ is singular at some point on $\phi_L^{-1}(p_1)$. Hence, there exists an irreducible subset $F_{p_1} \subset \phi_L^{-1}(p_1)$ such that  the closure in $\cD$ of the union when $p_1' \in F_{p_1}$ of the linear systems $|L-2p_1'|$ is an irreducible subvariety $\cD_{p_1}' \subset \cD$ containing $\cD_{p_1}$:
$$\cD_{p_1} \subset \cD_{p_1}' =\overline{\cup_{p_1' \in F_{p_1}}|L-2p_1'|} \subset \cD$$
By Lemma \ref{lemma:lm1summary} and in view of (\ref{eq:dimDp1}) one gets the following possibilities:
\begin{enumerate}
\item either $\dim F_{p_1}=0$, $F_{p_1}=p_1'$ and
\begin{enumerate}
\item $\cD_{p_1} \subset \cD_{p_1}' =|L-2p_1'|=\P^{N-n+1}$, or
\item  $\cD_{p_1}=\cD_{p_1}' =|L-2p_1'|=\P^{N-n}$; or 
\end{enumerate}
\item $\dim F_{p_1}=1$, and
\begin{enumerate}
\item $\cD_{p_1} \subset \cD_{p_1}'$, $\dim \cD_{p_1}'=N-n+1$, or
\item  $\cD_{p_1} \subset  \cD_{p_1}'=|L-2p_1'|=|L-2p_1''|=\P^{N-n}$, for any $p_1',p_1'' \in F_{p_1}$; or
\end{enumerate}
\item $\dim F_{p_1}=2$, $\cD_{p_1} \subset |L-2p_1'|=|L-2p_1''|=\P^{N-n}$, for any $p_1',p_1'' \in F_{p_1}$
\end{enumerate}

%

In cases (1.a), (2.a), and (3) $\cD'_{p_1}$ is an irreducible component of $\cD$ and then one can repeat the procedure  for any other general $q_1 \in C_1$ to get $$\cD'_{p_1}=\cD'_{q_1},$$ which is a contradiction since the general hyperplane in $\cD_{p_1}$ is singular in a linear space $\langle V,p_1,p_2\rangle$, which does not contain $C_1$, see (\ref{eq:joinbound}).

In cases (1,b) and (2.b), for $y \in Y_0$, the set of hyperplanes containing $T_{\phi_L(X),y}$ is contained in a linear space of dimension $N-n$, which is the linear space of hyperplanes containing a fixed $T=\P^{n-1}$. Hence one can use (\ref{eq:terracini}) to show that for the general $p_2 \in C_2$  the set of hyperplanes containing $T_{\phi_L(X),y}$ is contained in the set of hyperplanes containing $T$, which leads to:
$$\P^{n-1}=T \subset \cap_{p_2 \in C_2}\langle V,T_{C_1,p_1},T_{C_2,p_2} \rangle, $$ and consequently to
$$\P^{n-2} \subset \cap_{p_2 \in C_2}\langle V,p_1,T_{C_2,p_2} \rangle$$ But this implies that $Y$ is a cone with vertex a linear space of dimension $n-2$, and hence $Y$ is linear, contradicting (\ref{eq:joinbound}).

Let us consider Case II. Take $p \in V$ a general point. Since any hyperplane section in $\phi_L(X)^\vee$ is singular along $V$, then the corresponding element in $\cD$ is singular at some point in $\phi_L^{-1}(v)$. In a similar way that in Case I, consider the singular points on the fiber  $\phi_L^{-1}(v)$ of hyperplane sections corresponding to smooth points of $\phi_L(X)^{\vee}$ . This leads to an irrreducible $F_v \in \phi_L^{-1}(v)$ such that $\phi_L(X)^\vee$ is contained in (in fact equal because it is of maximal dimension) $\cD_v$ the closure in $\cD$ of the union of the linear systems $|L-2p'|$,  $p' \in F_v$. Since $\dim \phi_L(X)^\vee=N-n+1$ and $\phi_L(X)^\vee \subset |L|$ is not linear then $\dim F_v=1$ and $\dim |L-2p'|=N-n$ for all $p' \in F_v$. In fact, $\dim F_v=0$ or $2$, or $\dim F_v=1$ and $|L-2p'|=|L-2q'|$, for all $p',q' \in F_v$ would imply $\phi_L(X)^\vee$ linear. Consider $p' \in F_v$ and the linear space $\P^{N-n} =|L-2p'| \subset \phi_L(X)^\vee$. The general $H\in |L-2p'|$ is a smooth point of $\phi_L(X)^\vee$, hence it is singular along a linear space spanned by the vertex and a point $s_H \in S$ where $\P^n=\langle T_{S,s_H}, V\rangle \subset H$. Since $\dim |L-2p'|>N-1-n$ then $s_H \ne s_{H'}$ for $H,H' \in |L-2p'|$ and one can collect the points $s_H$ in an irreducible curve $C_{p'}$ to construct a cone $Y_{p'}$ with vertex $V$ over $C_{p'}$, whose dimension is $n-1$. For the dual of that cone we get that $Y_{p'}^\vee \subset |L-2p'|=\P^{N-n}$ and then, by biduality, there exists a $\P^{n-1}$ contained in $Y_{p'}$, which is showing that $Y_{p'}=\P^{n-1}$, and $C_{p'}$ is a line. Two different general points $p' \ne q'$ in $F_{v}$ lead to different linear spaces $\P^{n-1}_{p'}\ne \P^{n-1}_{q'}$ and then to a positive dimensional family of linear spaces of dimension $n-1$ on $\phi_L(X)$, which finally provides the structure of generalized scroll over a line. Consider $S'$ a general section of $\phi_L(X)$ by $n-2$ hyperplanes not containing $V$, which is then swept out by a one dimensional  family of lines. Since $S'$ is dominated by $S$, then $S'$ is rational. Since there is a line trough a general point of $S'$, then it is ruled in lines. Hence there exists a rational map $r': S' \to \P^1$ such that the fibers are the lines, and then a structure of generalized scroll over a line for $\phi_L(X)$.

\end{proof}

\begin{example}\label{ex:nef} The different possibilities appearing in the proof of Theorem \ref{thm:big} are effective. Consider for instance a rank two projective bundle $\P_{\P^1}(\cO \oplus \cO(r_1)\oplus \cO(r_2))$, see Example \ref{ex:projectivebundles}. If $r_1\ne 0$ then the contraction of the minimal section is providing a map onto a cone of vertex a point over the ruled surface $\mathbb{P}_{\mathbb{P}^1}(\mathcal{O}(r_1)\oplus \mathcal{O}(r_2))$. If $r_1=0$ we get a morphism to a cone with vertex a line over a rational normal curve of degree $r_2$. If, in particular, $r_2=1$, this is nothing but the contraction of the exceptional divisor of the blow up of $\mathbb{P}^3$ along a line. Finally, with the notation of Example \ref{ex:lbhirz} take $X=\mathbb{P}_{\mathbb{P}^1 \times \mathbb{P}^1}(\mathcal{O}(sD_0') \oplus \mathcal{O}(tD_3'))$, with $s,t>0$, the tautological bundle $L$ defines a map onto the join of two rational normal curves of degree $s$ and $t$. In this example $c_3(J_1(L))=2(t(s-1)+s(t-1))$, which  only vanishes when $s=t=1$, that is, when $\phi_L(X)=\P^3$.
\end{example}

Let us end this section with the case $k=N-2$.

\begin{proposition}\label{prop:N-2}  Let  $(X,L)$ be as in Problem \ref{problem}. If  $k=N-2$  then $\phi_L:X \to \P^{N}$ is surjective.
\end{proposition}

\begin{proof} By Lemma \ref{lemma:lm1summary}  it holds that $k \leq \min\{n,N\}$. The cases $k=N-2 \geq  n-1$ or $k=N-2=n-2$ and $L$ not big are not possible, see Lemma \ref{lema:bound} and Proposition \ref{prop:notbig}. If $k=N-2=n-2$ and $L$ is big, then we get in this case the surjectivity of $\phi_L:X \to \P^{N}$. Finally, if $k=N-2<n-2$ and $\phi_L:X \to \P^{N}$ is not surjective then $N-2<\dim \phi_L(X)<N$, i. e., $\dim\phi_L(X)=N-1$, and by Lemma \ref{lemma:phi}(iii) $\phi_L(X)$ is a cone over a curve with vertex in a linear space of dimension $N-3$, contradicting
Lemma \ref{lemma:coneoveracurve}.
\end{proof}

\begin{remark} Let us refer here to Lemma \ref{lema:bound} (iv) and Remark (\ref{rem:(4)}) for examples of this situation.
\end{remark}

\section{Low degree dual varieties}\label{sec:ld}

The degree of the dual variety of a smooth projective variety $X \subset \P^N$ is a classical invariant studied from different perspectives. If the embedding is degenerate, that is, there exists a linear space $\P^M$ of dimension $M<N$ containing $X$, the dual variety is a cone over the dual variety of $X \subset \P^M$, see \cite[Thm.~1.23]{Tev}. On the other hand, if $X \subset \P^N$  is not linearly normal then its dual variety is the corresponding linear section of the dual variety of the linearly nomal embedding $X \subset \P^M$, $M>N$, cf. \cite[Thm.~1.21]{Tev}. These two remarks show that, in order to study the degre of the dual variety of $X \subset \P^N$ one can focus in non-degenerate and linearly normal embeddings. Assuming this and taking the biduality ($X^{\vee \vee}=X$, see \cite[Thm.~1.2]{Tev}) into account, the dual variety cannot be linear and it is a quadric if and only if $X \subset\P^N$ is a smooth quadric. There exists a classification when the degree of the dual variety is three, see \cite[Thm.~5.2]{zak}. If $X$ is toric, we can write:

\begin{theorem}\label{thm:codegree}\cite[Thm.~5.2]{zak} Let $X \subset \P^N$ be  a smooth, irreducible, non-degenerate projective toric variety and $L=\cO_{\P^N}(1)|_X$. Assume moreover that the embedding is linearly normal. If $ \deg(X^\vee)\leq 3$ then $(X,L)$ is one of the following:
\begin{enumerate}
\item[(i)] $X=\P^n$, $n=1,2$,  $L=\cO(2)$,  $\deg(X^\vee)=2$ if $n=1$ and $\deg(X^\vee)=3$ if $n=3$,  see \rm{(\ref{eq:proj})}.
\item[(ii)] $X=\P^{n-m} \times \P^m$, $n-m \geq m$, $L=D_0+D_{n+1}$ where 
$$\begin{cases}
n=2, m=1, \quad \deg(X^\vee)=2,\hbox{ see \rm{(\ref{eq:products})}}\\
n=4, m=2, \quad \deg(X^\vee)=3, \hbox{ see \rm{(\ref{eq:products})}} \\
n=3, m=1, \quad \deg(X^\vee)=3, \hbox{ see \rm{(\ref{eq:segre})}}
\end{cases}$$
\item[(iii)] $X=\P_{\P^1}(\cO \oplus\cO(1))$, $L=D_0+D_3$, $\deg(X^\vee)=3$, see \rm{(\ref{eq:blowup})}. 
\end{enumerate}
\end{theorem}

The next case to consider is then $\deg(X^\vee)=4$. For this purpose one can read the computations of the examples in Section \ref{sec:examples} in terms of the degree of the dual variety. Starting with the case $X=\P^n$, see (\ref{eq:proj}), the only possibility of having degree $4$ is $n=4$, $m=2$.

By inspection of formulae (\ref{eq:rankoneproducts}) and (\ref{eq:rankone}) one gets the following:

\begin{proposition}\label{prop:rankone}
Consider $X=\P_{\P^{n-1}}(\cO\oplus\cO(r_1))$ as in Example \ref{ex:projectivebundles}, $L=aD_0+bD_3$ a very ample line bundle on $X$ defining a linearly normal embedding $X \subset |L|^\vee=\P^N$. Let $d$ be the degree of $X^\vee\subset \P^{N\vee}$.  Then $d=4$ if and only if one of the following holds:
\begin{enumerate}
\item[(i)] $n=2$, $r_1=2$, $(a,b)=(1,1)$; 
\item[(ii)] $n=2$, $r_1=0$, $(a,b)=(2,1)$ or $(1,2)$; 
\item[(iii)] $n=3$, $r_1=1$, $(a,b)=(1,1)$; 
\item[(iv)] $n=4$, $r_1=0$, $(a,b)=(1,1)$. 
\end{enumerate}
\end{proposition}



By inspection of formulae (\ref{eq:ranktodim3}) and (\ref{eq:ranktodim3def}), we have:

\begin{proposition}\label{prop:rank2}
Consider $X=\P_{\P^1}(\cO \oplus \cO(r_1) \oplus \cO(r_2))$  with $0 \leq r_1 \leq r_2$, and let $L$ be a very ample line bundle on $X$ defining a linearly normal embedding $X \subset \P^N$. With the notations of the previous proposition, it holds that $d=4$ if and only if $r_1=0$, $r_2=1$ and $a=b=1$.

\end{proposition} 
By inspection of formulae (\ref{eq:projhirz}) and  (\ref{eq:projhirz2}), we get:

\begin{proposition}\label{prop:hirz}
Let $X=\P_{\P^1}(\cO \oplus \cO(r))$ be a Hirzebruch surface and let $Y$ be a decomposable line bundle on $X$ as in Example \ref{ex:lbhirz}. Let $L$ be a very ample line bundle on $Y$ defining a linearly normal embedding $Y \subset \P^N$. Then, with the notation above, $d=4$ if and only if $r=s=t=0$, $a=b=1$. 
%
\end{proposition}

We have also computed the minimal degree $d$ of the dual variety for all of the smooth Fano toric 3-folds, classified in \cite{bat}, \cite{WW} (see Figure \ref{fig:fano3}), obtaining the following results for low degrees:

\begin{proposition}\label{prop:fano} Let $X$ be a smooth toric Fano 3-fold and let $L$ be a very ample line bundle on $X$ defining an embedding $X \subset \P^n$. Then, either $d=3$ and $X$ is the case recalled in (iii) of Proposition \ref{prop:rankone}, or $d=4$ and $X$ is that of Proposition \ref{prop:rank2}, or $X=\P^1\times\P^1\times\P^1$ and $L=\pi_1^*(\cO(1))+\pi_2^*(\cO(1))+\pi_3^*(\cO(1))$, where $\pi_i$, $i=1,2,3$ is the projection onto the $i$-factor. For the rest of the cases $d>4$. 
\end{proposition} 

\begin{proof} We will follow the notation for the smooth toric Fano 3-folds in \cite[Thm.~1]{WW}. If the Picard rank $\rho$ is less or equal to 4, for a suitable basis of the Picard group the nef cone is generated by the standard $\Z$-basis $e_1\ldots,e_\rho$. Then, by Proposition \ref{prop:mindeg}, the lowest degree is obtained for the value of $c_3(J_1(L))$ at $(1,\ldots,1)$. In all cases except for $X=\P^2\times\P^1$ and $X=\P_{\P^1}(\cO\oplus\cO\oplus\cO(1))$ this value is greater than zero, representing the minimal degrees shown in Figure \ref{fig:fano3}. In the case $X=\P^2\times\P^1$ with rays 
\[
\Sigma_X=\{(1,0,0),(0,1,0),(-1,-1,0),(0,0,1),(0,0,-1)\},
\]
the Nef cone is generated by $D_0$ and $D_3$. Then the value of $c_3(J_1(L))$ for $L=aD_0+bD_3$ is in fact zero when $a=1$ so in this case the degree is $d=c_2(J_1(L))L=3b$, obtaining the lowest value 3 when $L=D_0+D_3$. Similarly, for the case $X=\P_{\P^1}(\cO\oplus\cO\oplus\cO(1))$ with rays 
\[
\Sigma_X=\{(1,0,0),(0,1,0),(-1,-1,0),(0,0,1),(0,1,-1)\},
\]
the nef cone is generated by $D_0$ and $D_3$. The value of $c_3(J_1(L))$ for $L=aD_0+bD_3$ is again zero when $a=1$ so the degree in this cases is $d=c_2(J_1(L))L=3b+1$. This time the lowest value is 4 when $L=D_0+D_3$.  

Let us discuss in more detail the two smooth toric Fano 3-folds of Picard rank $\rho=5$, since their nef cones are not simplicial. The first case is $X\cong DS_6\times\P^1$ where $DS_6$ denotes the blow-up of $\P^2$ in 3 points. It is defined by the rays 
\[
\Sigma_X=\{(1,0,0),(1,0,1),(0,0,1),(-1,0,0),(-1,0,-1),(0,0,-1),(0,1,0),(0,-1,0)\}
\]
Taking $\Pic X=<D_0,D_1,D_2,D_3,D_7>$, the nef cone in these coordinates is generated by $v_1 = (0,0,1,1,0)$, $v_2 = (0,1,1,0,0)$, $v_3 = (1,1,1,0,0)$, $v_4 = (0,0,0,0,1)$, $v_5 = (1,1,0,0,0)$, $v_6 = (0,1,1,1,0)$. By using Proposition \ref{prop:mindeg} and applying the formula of the top Chern class for every point in the set $A$, the lowest degree is 24 and it is attained at $v=(1,2,2,1,1)$. This vector is interior to the nef cone, thus ample, and corresponds to the divisor $L=D_0+2D_1+2D_2+D_3+D_7$. 

In the second case, $X\cong F_1^5$ is generated by the rays 
\[
\Sigma_X=\{(1,0,0),(1,0,1),(0,0,1),(-1,0,0),(-1,0,-1),(0,0,-1),(0,1,0),(1,-1,0)\}
\]
Taking $\Pic = <D_0,D_1,D_2,D_3,D_7>$, the nef cone in these coordinates is generated by $v_1 = (0,0,1,1,0)$, $v_2 = (0,1,1,0,0)$, $v_3 = (1,1,1,0,1)$, $v_4 = (0,0,0,0,1)$, $v_5 = (1,1,0,0,1)$, $v_6 = (0,1,1,1,0)$. Then, using again Proposition \ref{prop:mindeg} and applying the top Chern class formula, the point of the set $A$ which gives this time the lowest degree is $v=(1,2,2,1,2)$. As in the previous case, this vector corresponds to an ample divisor, namely $L=D_0+2D_1+2D_2+D_3+2D_7$, which gives the lowest degree $c_3(J_1(L))=72$.
\end{proof}

We conclude by showing in Figure \ref{fig:fano3} the miminal degrees $d$ of the discriminants of the 18 smooth toric Fano 3-fods for any possible ample line bundle $L$. We recall that $DS_8$, $DS_7$ and $DS_6$ are the del Pezzo surfaces corresponding to the blow-up of 1, 2 and 3 points of $\P^2$ respectively. For completeness, we also show in Figure \ref{fig:fano2} the 5 smooth toric Fano surfaces, where only $DS_6$ has a non-simplicial nef cone.

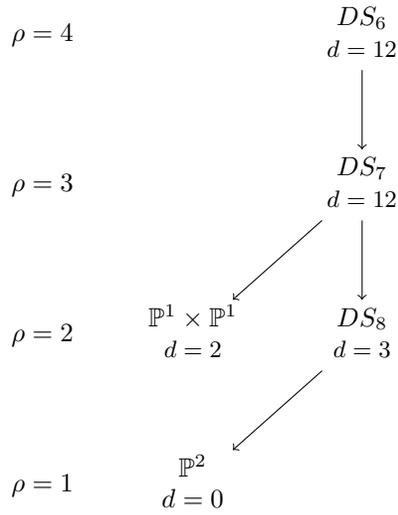
\begin{figure}
\begin{tikzpicture} 
\node at (0,0) {$\rho=1$};\node at (0,2) {$\rho=2$};\node at (0,4) {$\rho=3$};\node at (0,6) {$\rho=4$};
\node (b) at (2,0) [inner sep=1.25pt] {$\begin{array}{c}\P^2\\d=0\end{array}$};
\node (f) at (2,2) [inner sep=1.25pt] {$\begin{array}{c}\P^1\times\P^1\\\text{\small $d=2$}\end{array}$};
\node (g) at (4.25,2) [inner sep=1.25pt] {$\begin{array}{c}DS_8\\\text{\small $d=3$}\end{array}$};
\node (m) at (4.25,4) [inner sep=1.25pt] {$\begin{array}{c}DS_7\\\text{\small $d=12$}\end{array}$};
\node (q) at (4.25,6) [inner sep=1.25pt] {$\begin{array}{c}DS_6\\\text{\small $d=12$}\end{array}$};
\draw[->] (g) -- (b);
\draw[->] (m) -- (f);\draw[->] (m) -- (g);
\draw[->] (q) -- (m);
\end{tikzpicture} 
\caption{The minimal degrees $d$ for the discriminants of the 4 smooth smooth toric Fano surfaces for $L$ ample and their birational relations.}
\label{fig:fano2}
\end{figure}

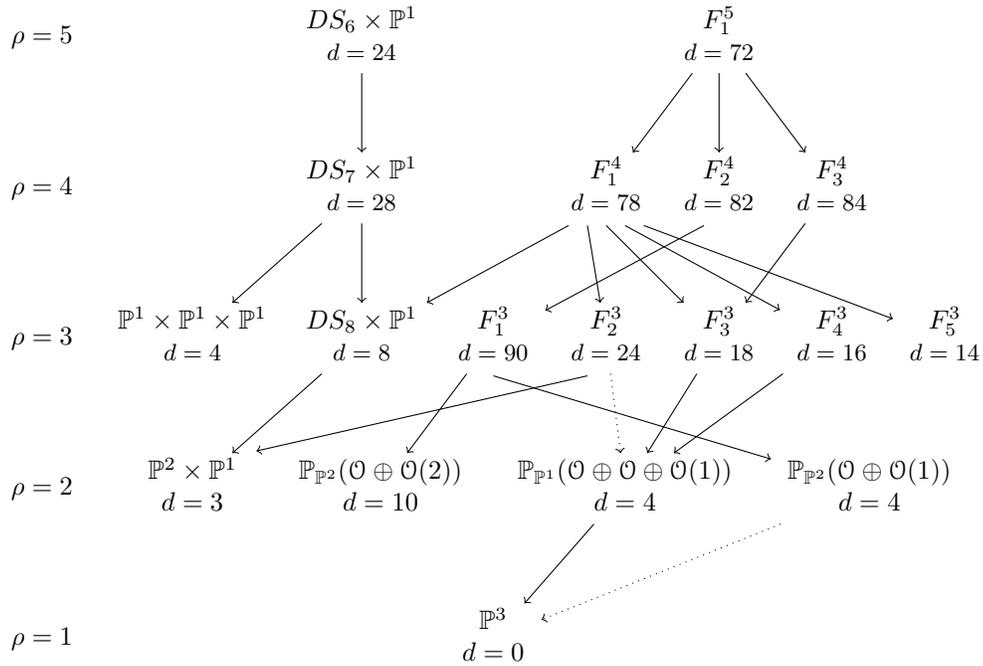
\begin{figure}
\begin{tikzpicture} 
\node at (0,-0.5) {$\rho=1$};\node at (0,1.5) {$\rho=2$};\node at (0,3.5) {$\rho=3$};\node at (0,5.5) {$\rho=4$};\node at (0,7.5) {$\rho=5$};
\node (a) at (6,-0.5) [inner sep=1.25pt] {$\begin{array}{c}\P^3\\d=0\end{array}$};
\node (b) at (2,1.5) [inner sep=1.25pt] {$\begin{array}{c}\P^2\times\P^1\\d=3\end{array}$};
\node (c) at (4.5,1.5) [inner sep=1.25pt] {$\begin{array}{c}\P_{\P^2}(\cO\oplus\cO(2))\\d=10\end{array}$};
\node (d) at (7.75,1.5) [inner sep=1.25pt] {$\begin{array}{c}\P_{\P^1}(\cO\oplus\cO\oplus\cO(1))\\d=4\end{array}$};
\node (e) at (11,1.5) [inner sep=1.25pt] {$\begin{array}{c}\P_{\P^2}(\cO\oplus\cO(1))\\d=4\end{array}$};

\node (f) at (2,3.5) [inner sep=1.25pt] {$\begin{array}{c}\P^1\times\P^1\times\P^1\\\text{\small $d=4$}\end{array}$};
\node (g) at (4.25,3.5) [inner sep=1.25pt] {$\begin{array}{c}DS_8\times\P^1\\\text{\small $d=8$}\end{array}$};
\node (h) at (6,3.5) [inner sep=1.25pt] {$\begin{array}{c}F^3_1\\\text{\small $d=90$}\end{array}$};
\node (i) at (7.5,3.5) [inner sep=1.25pt] {$\begin{array}{c}F^3_2\\\text{\small $d=24$}\end{array}$};
\node (j) at (9,3.5) [inner sep=1.25pt] {$\begin{array}{c}F^3_3\\\text{\small $d=18$}\end{array}$};
\node (k) at (10.5,3.5) [inner sep=1.25pt] {$\begin{array}{c}F^3_4\\\text{\small $d=16$}\end{array}$};
\node (l) at (12,3.5) [inner sep=1.25pt] {$\begin{array}{c}F^3_5\\\text{\small $d=14$}\end{array}$};

\node (m) at (4.25,5.5) [inner sep=1.25pt] {$\begin{array}{c}DS_7\times\P^1\\\text{\small $d=28$}\end{array}$};
\node (n) at (7.5,5.5) [inner sep=1.25pt] {$\begin{array}{c}F^4_1\\\text{\small $d=78$}\end{array}$};
\node (o) at (9,5.5) [inner sep=1.25pt] {$\begin{array}{c}F^4_2\\\text{\small $d=82$}\end{array}$};
\node (p) at (10.5,5.5) [inner sep=1.25pt] {$\begin{array}{c}F^4_3\\\text{\small $d=84$}\end{array}$};

\node (q) at (4.25,7.5) [inner sep=1.25pt] {$\begin{array}{c}DS_6\times\P^1\\\text{\small $d=24$}\end{array}$};
\node (r) at (9,7.5) [inner sep=1.25pt] {$\begin{array}{c}F^5_1\\\text{\small $d=72$}\end{array}$};

\draw[->] (d) -- (a);\draw[->,dotted] (e) -- (a);
\draw[->] (g) -- (b);
\draw[->] (h) -- (c);\draw[->] (6,3) -- (e);
\draw[->,dotted] (i) -- (d);
\draw[->] (7.25,3) -- (2.85,2);\draw[->] (j) -- (d);\draw[->] (k) -- (d);
\draw[->] (m) -- (f);\draw[->] (m) -- (g);
\draw[->] (7,5) -- (g);\draw[->] (7.25,5) -- (i);\draw[->] (7.5,5) -- (j);\draw[->] (7.75,5) -- (k);\draw[->] (8,5) -- (l);
\draw[->] (8.8,5) -- (h);\draw[->] (p) -- (j);
\draw[->] (q) -- (m);\draw[->] (r) -- (n);\draw[->] (r) -- (o);\draw[->] (r) -- (p);
\end{tikzpicture} 
\caption{The 18 smooth smooth toric Fano 3-folds with their minimal degrees $d$ of the discriminant for $L$ ample, and their relations following the notation in \cite{WW}. The arrow indicates blow-up with center a line and the dotted arrow indicates a blow-up with center a point.}
\label{fig:fano3}
\end{figure}


\bibliographystyle{plain}
\bibliography{biblio}

\end{document}